\documentclass[11pt,a4paper]{amsart}
\usepackage{amssymb,amsmath,epsfig,graphics,mathrsfs,enumerate,verbatim}
\usepackage[pagebackref,colorlinks=true,linkcolor=blue,citecolor=blue]{hyperref}
\usepackage{fancyhdr}
\usepackage{hyperref}
\hypersetup{
 colorlinks   = true,
 urlcolor     = blue,
 linkcolor    = blue,
 citecolor   = red ,
 bookmarksopen=true
}

\usepackage{amsmath}
\usepackage{amsfonts}
\usepackage{amssymb}
\usepackage{amsthm}
\usepackage{epsfig,graphics,mathrsfs}
\usepackage{graphicx}
\usepackage{dsfont}

\usepackage[usenames, dvipsnames]{color}

\usepackage{hyperref}

\def \phi {\varphi}

\def \R {\mathbb{R}}

\def \vf{\varphi}

\def \So {\mathscr{S}(\Rm)}

\newcommand{\Rm}{\mathbb R^m}

\newcommand{\p}{\partial}

\newcommand{\la}{\lambda}

\numberwithin{equation}{section}

\newcommand{\beq}{\begin{equation}}
\newcommand{\bea}[1]{\begin{array}{#1} }
\newcommand{\eeq}{ \end{equation}}
\newcommand{\ea}{ \end{array}}

\newcommand{\sa}{\langle}
\newcommand{\da}{\rangle}

\newtheorem{theorem}{Theorem}[section]
\newtheorem{lemma}[theorem]{Lemma}
\newtheorem{proposition}[theorem]{Proposition}
\newtheorem{corollary}[theorem]{Corollary}
\newtheorem{remark}[theorem]{Remark}

\numberwithin{equation}{section}

\begin{document}

\title[Uncertainty principles,  etc.]{Uncertainty principles for the imaginary Ornstein-Uhlenbeck operator}

\dedicatory{I dedicate this paper to Carlos Kenig,  long-time friend and deeply influential mathematician, on the occasion of his birthday}



\date{}

\begin{abstract}
We prove two forms of uncertainty principle for the Schr\"odinger group generated by the Ornstein-Uhlenbeck operator. As a consequence, we derive a related (in fact, equivalent) result for the imaginary harmonic oscillator. 
\end{abstract}

\author{Nicola Garofalo}

\address{Dipartimento d'Ingegneria Civile e Ambientale (DICEA)\\ Universit\`a di Padova\\ Via Marzolo, 9 - 35131 Padova,  Italy}
\vskip 0.2in
\email{nicola.garofalo@unipd.it}

\thanks{The author is supported in part by a Progetto SID (Investimento Strategico di Dipartimento): ``Aspects of nonlocal operators via fine properties of heat kernels", University of Padova (2022); and by a PRIN (Progetto di Ricerca di Rilevante Interesse Nazionale) (2022): ``Variational and analytical aspects of geometric PDEs". He is also partially supported by a Visiting Professorship at the Arizona State University}

\maketitle


\section{Introduction}\label{S:intro}

Given $m\in \mathbb N$, let $L$ indicate the Ornstein-Uhlenbeck operator in $\Rm$ defined by 
\[
L \vf = \Delta \vf - \sa x,\nabla \vf\da,
\]
see \cite{OU}.
As it is well-known, see \cite{Bo}, the invariant measure for $L$ is $d\gamma(x) = e^{-\frac{|x|^2}2} dx$. The operator $L$ is symmetric with respect to $d\gamma$, i.e., for any $\vf, \psi\in C^\infty_0(\Rm)$,
\[
\int_{\Rm} \vf L\psi d\gamma = \int_{\Rm} \psi L\vf d\gamma.
\]
Since $L$ is self-adjoint in $L^2(\Rm,d\gamma)$, by Stone's theorem, see \cite[Theorem 1, p. 345]{Yo}, there exists a strongly-continuous group $e^{i t L}$ of unitary operators in $L^2(\Rm,d\gamma)$. Such group provides the solution operator for the Cauchy problem in $\Rm\times (0,\infty)$ for the Schr\"odinger equation 
\begin{equation}\label{cp0}
\begin{cases}
\p_t f - i L f = 0,
\\
f(x,0) = \vf(x).
\end{cases}
\end{equation}
Notably, the differential equation in \eqref{cp0} is invariant with respect to the (complex) left-translation 
\[
\tau_{(x,s)} (y,t) = (x,s)\circ (y,t) = (y + e^{it}x,t+s),
\]
in the sense that $(\p_t  - i L)(\tau_{(x,s)} f) = \tau_{(x,s)}[(\p_t  - i L) f]$. It is worth emphasising here that \eqref{cp0} represents a basic model for a more general class of (possibly) degenerate operators of interest in mathematics and physics, introduced in the work \cite{Ho}. 

The main objective of this note is the following form of uncertainty principle for the group $e^{i t L}$. When $\vf\in L^2(\Rm,d\gamma)$, we will write $f(x,t) = e^{itL}\vf(x)$.  

\begin{theorem}\label{T:main}
Assume that for some $a, b>0$ one has
\begin{equation}\label{L2}
||e^{a|\cdot|^2} f(\cdot,0)||_{L^2(\Rm, d\gamma)} + ||e^{b|\cdot|^2}  f(\cdot,s)||_{L^2(\Rm,d\gamma)}<\infty.
\end{equation}
If $a b \sin^2 s  \ge \frac{1}{16}$, then $f\equiv 0$ in $\Rm\times \R$. 
\end{theorem}

We note explicitly that the assumption $a b \sin^2 s\ge \frac{1}{16}$ automatically excludes the possibility that $s = k \pi$, with $k\in \mathbb Z$. This discrete set of points is where the covariance matrix $Q(t) I_m$, with $Q(t)$ defined in \eqref{Qt0} below, becomes singular and the representation \eqref{OUim} of $e^{itL}$ ceases to be valid. The proof of Theorem \ref{T:main} combines our formula \eqref{final} in Proposition \ref{P:semigroup} with the interesting $L^2$ version of the classical theorem of Hardy  for the Fourier transform due to Cowling, Escauriaza, Kenig, Ponce and Vega in \cite{CEKPV}. 

Theorem \ref{T:main} implies (and is in fact equivalent to) the following uncertainty principle for the Schr\"odinger group $e^{itH}$ associated with the harmonic oscillator $H = \Delta - \frac{|x|^2}4$, \footnote{This operator is usually defined as $H = \Delta - |x|^2$. We are using the $1/4$ normalisation in order not to have to change in $\Delta - 2 \sa v,\nabla\da$ that of the Ornstein-Uhlenbeck operator in the statement of Proposition \ref{P:connect} below.} representing the solution operator for the Cauchy problem
\begin{equation}\label{cp1}
\begin{cases}
\p_t u - i H u = 0,
\\
u(x,0) = u_0(x).
\end{cases}
\end{equation}
Again by Stone's theorem, there exists a strongly-continuous group $e^{i t H}$ of unitary operators in $L^2(\Rm)$. If for $u_0\in L^2(\Rm)$ we let $u(x,t) = e^{itH} u_0(x)$, then we have the following result. 

\begin{corollary}\label{C:main}
Assume that for some $a, b>0$ one has
\begin{equation}\label{L22}
||e^{a|\cdot|^2} u(\cdot,0)||_{L^2(\Rm)} + ||e^{b|\cdot|^2} u(\cdot,s)||_{L^2(\Rm)}<\infty.
\end{equation}
If $a b \sin^2 s  \ge \frac{1}{16}$, then $u\equiv 0$ in $\Rm\times \R$. 
\end{corollary}

The passage from Theorem \ref{T:main} to Corollary \ref{C:main} (and vice-versa) is based on Proposition \ref{P:connect} below. We note that combining Proposition \ref{P:semigroup} with the uncertainty principle for the Fourier transform due to Hardy, see \cite{Ha} or also \cite[Theorem 1.2]{CEKPV}, we obtain corresponding $L^\infty$ versions of Theorem \ref{T:main} and Corollary \ref{C:main}. For the Schr\"odinger group $e^{itL}$ we have the following. 

\begin{theorem}\label{T:main2}
Suppose that for some $C, a, b>0$ one has for any $x\in \Rm$
\begin{equation}\label{Linfty}
|e^{-\frac{|x|^2}4}f(x,0)| \le C e^{-a|x|^2},\ \ \ \ \ \  |e^{-\frac{|x|^2}4} f(x,s)| \le  C e^{-b|x|^2}.
\end{equation}
If $a b \sin^2 s  \ge \frac{1}{16}$, then $f\equiv 0$ in $\Rm\times \R$. 
\end{theorem}

We mention that, for the harmonic oscillator $H$, an $L^\infty$ uncertainty principle such as Theorem \ref{T:main2}, in which the Gaussian function $e^{-\frac{|x|^2}4}$ does not appear in the hypothesis \eqref{Linfty}, was found in Theorem 6.2 of the interesting paper \cite{Vel}. Uncertainty inequalities is a vast subject and there exist many beautiful and important results scattered in the literature which is impossible to quote here. We refer the reader to \cite{Folland} and \cite{Veluma} for an interesting account which covers up to 2004. 

In closing, we state an optimal dispersive estimate that we obtain combining Proposition \ref{P:semigroup} with Beckner's deep improvement of the Hausdorff-Young inequality, see \cite{Be}. For the definition of the class $\mathscr K(\Rm)$ see \eqref{K} below.
\begin{proposition}\label{P:disp}
Let $\vf\in \mathscr K(\Rm)$, and $f(x,t) = e^{itL}\vf(x)$ with $t \not= k \pi$, with $k\in \mathbb Z$. For any $1\le p\le 2$ one has
\begin{equation}\label{disp}
||e^{-\frac{|\cdot|^2}4} f(\cdot,t)||_{L^{p'}(\Rm)} \le \left(\frac{p^{1/p}}{{p'}^{1/p'}}\right)^{\frac m2} \ \frac{1}{(4\pi |\sin t|)^{m(\frac 12 - \frac 1{p'})}}\ ||e^{-\frac{|\cdot|^2}4} \vf||_{L^p(\Rm)},
\end{equation}
where $1/p + 1/{p'} = 1$. The estimate \eqref{disp} is optimal, in the sense that it cannot possibly hold in the range $2<p\le \infty$. Moreover, equality is attained in \eqref{disp} by initial data $\vf(x) =  e^{-\alpha |x|^2+ \frac{|x|^2}4}$, $\Re \alpha>0$.  
\end{proposition}
Note that as $t\to 0$ the estimate \eqref{disp} displays the same behaviour $t^{-m(\frac 12 - \frac{1}{p'})}$ as that of the free Schr\"odinger group $e^{i t \Delta}$ (for the latter, see e.g. \cite[Lemma 1.2]{GV}). 
It may be of interest to compare \eqref{disp} with the following dispersive estimate for the Schr\"odinger equation with friction
\begin{equation}\label{OUrd}
\p_t v - i \Delta v + \sa x,\nabla v\da = 0.
\end{equation}
The PDE \eqref{OUrd} is very different from \eqref{cp0}, as one can readily surmise from its invariance group of (real) left-translations 
\[
\sigma_{(x,s)} (y,t) = (x,s)\circ (y,t) = (y + e^{t}x,t+s).
\] 
As a special case of the result in \cite[Theorem 4.1]{GL}, one obtains for the semigroup $\{\mathcal T(t)\}_{t\ge 0}$ associated with the Cauchy problem for \eqref{OUrd}
\begin{equation}\label{feOU}
||\mathcal T(t) \vf||_{L^{p'}(\Rm)} \le C(m,p)\ \frac{e^{\frac{m}{p'} t}}{(1-e^{-2t})^{m(\frac 12 - \frac 1{p'})}} ||\vf||_{L^p(\Rm)}.
\end{equation}
The behaviour as $t\to 0^+$ in \eqref{disp} and \eqref{feOU} is the same, but because of the presence of $e^{-\frac{|\cdot|^2}4}$ in the former, the two estimates are incomparable.

\section{The imaginary Ornstein-Uhlenbeck}\label{S:senzadrift}

In this section we solve the Cauchy problem \eqref{cp0} by constructing a representation of the solution operator. Our main result is \eqref{final} in Proposition \ref{P:semigroup}. It rests on Proposition \ref{P:OUim}, which could be derived from the physicists' Wick rotation, see \cite[Section 3]{Wi}, in the Mehler formula for the harmonic oscillator $H = - \Delta + \frac{|x|^2}4$, see e.g. p.154 in \cite{BGV} or p.55 in \cite{Folland}. Such derivation however requires justifying some nontrivial passages. Our elementary construction makes the solution operator \eqref{OUim}  immediately available independently from the harmonic oscillator and also leads to directly unveil the basic identity \eqref{final}.  

In what follows, given a function $\vf$ we indicate with $\psi$ the function
\begin{equation}\label{psi}
\psi(x) = e^{- \frac{|x|^2}{4}} \vf(x),\ \ \ \ \ \ x\in \Rm.
\end{equation} 
Notice from \eqref{psi} that $\psi\in L^2(\Rm) \Longleftrightarrow \vf\in L^2(\Rm,d\gamma)$, and
\begin{equation}\label{phipsi}
||\psi||_{L^2(\Rm)} = ||\vf||_{L^2(\Rm,d\gamma)}.
\end{equation}
We denote by 
\begin{equation}\label{K}
\mathscr K(\Rm) = \{\vf\in C^\infty(\Rm)\mid \psi \in \mathscr S(\Rm)\}.
\end{equation}
It is clear that $\mathscr K(\Rm)$ is dense in $L^2(\Rm,d\gamma)$.

\begin{remark}\label{R:stable}
It follows from Proposition \ref{P:semigroup} below that
\begin{equation}\label{nice}
e^{i t L} : \mathscr K(\Rm)\ \longrightarrow\ \mathscr K(\Rm).
\end{equation}
\end{remark}
Throughout this note we let $J^+ = (0,\pi)$, $J^- = (\pi,2\pi)$, and denote 
\[
J = J^+\cup J^-. 
\]
The reader should keep in mind that the covariance matrix $Q(t) = e^{-it} \sin t\ I_m$ defined by  \eqref{Qt0} below is invertible for any $t\in J$. Since such matrix can be extended by periodicity to the whole of $\R\setminus \pi \mathbb Z$, we will confine the attention to the set $J$.

\begin{proposition}\label{P:OUim}
Let $\vf\in \mathscr K(\Rm)$. For every $x\in \Rm$ the function  
\begin{equation}\label{OUim}
f(x,t) = \begin{cases}
\frac{(4\pi
)^{-\frac{m}{2}} e^{\frac{imt}2}}{e^{\frac{i\pi m}{4}} (\sin t)^{\frac m2}} \int_{\Rm} e^{i \frac{|e^{it/2} y-e^{-it/2} x|^2}{4 \sin t}} \vf(y) dy,\ \ \ \ \ \ \ t\in J^+,
\\
\\
\frac{(4\pi
)^{-\frac{m}{2}} e^{\frac{imt}2}}{e^{\frac{i 3\pi m}{4}} |\sin t|^{\frac m2}} \int_{\Rm} e^{-i \frac{|e^{it/2} y-e^{-it/2} x|^2}{4 |\sin t|}} \vf(y) dy,\ \ \ \ \ \ \ t\in J^-,
\end{cases}
\end{equation} 
and $f(x,0) = \vf(x)$, solves \eqref{cp0} in $\Rm\times J$.
\end{proposition}

\begin{proof}

We begin with a simple, but critical observation. Suppose that $v$ and $f$ are connected by the relation
\begin{equation}\label{drift0}
v(x,t) = f(e^{i t} x,t).
\end{equation}
Then, $f$ is a solution of the Cauchy problem \eqref{cp0} if and only if $v$ solves the problem
\begin{equation}\label{cpGsenzadrift00}
\begin{cases}
\p_t v - i Q'(t) \Delta v = 0,
\\
v(x,0) = \vf(x),
\end{cases}
\end{equation} 
where we have let
\begin{equation}\label{Qt0}
Q(t) =  \int_0^t e^{-2is} ds = \frac{1-e^{-2it}}{2i} = e^{-it} \frac{e^{it}-e^{-it}}{2i} = e^{-it} \sin t.
\end{equation}
To prove that $v$ solves \eqref{cpGsenzadrift00}, we argue as follows. The chain rule gives from \eqref{drift0}
\[
v_t(x,t) = i e^{it} \sa x,\nabla f(e^{it} x,t)\da + f_t(e^{it} x,t).
\]
On the other hand, the PDE in \eqref{cp0} gives
\[
f_t(e^{it} x,t) = i \Delta f(e^{it} x,t) - i e^{it} \sa x,\nabla f(e^{it} x,t)\da.
\]
Combining the latter two equations, we infer that $v$ solves 
\[
v_t(x,t) = i \Delta f(e^{it} x,t).
\]
Next, differentiating \eqref{drift0} we find
\[
\Delta v(x,t) = e^{2it} \Delta f(e^{it} x,t).
\]
We thus conclude that 
\[
v_t(x,t) = i e^{-2it} \Delta v(x,t) = i Q'(t) \Delta v(x,t),
\]
where in the second equality we have used \eqref{Qt0}.
Summarising, the function $v$ solves the problem \eqref{cpGsenzadrift00}.
To find a representation formula for the latter, we use the Fourier transform.
Supposing that $v$ be a solution, we define
\begin{equation}\label{pFTv}
\hat v(\xi,t) = \int_{\Rm} e^{-2\pi i\sa \xi,x\da} v(x,t) dx.
\end{equation}
Then \eqref{cpGsenzadrift00} is transformed into
\begin{equation}\label{cpGsenzadrift0}
\begin{cases}
\p_t \hat v + 4\pi^2 i Q'(t)|\xi|^2 \hat v = 0,
\\
\hat v(\xi,0) = \hat \vf(\xi),
\end{cases}
\end{equation} 
whose solution is given by
\begin{equation}\label{hatv}
\hat v(\xi,t) = \hat \vf(\xi) e^{-4\pi^2 i Q(t) |\xi|^2}.
\end{equation} 
Note that, with $Q(t)$ as in \eqref{Qt0}, for every $t\in J$ the matrix $Q(t) I_m$ is invertible. Moreover, we have
\begin{equation}\label{iQ}
i Q(t) I_m = i(\cos t - i \sin t) \sin t I_m = \sin^2 t\ I_m + i \frac{\sin{2t}}2 I_m. 
\end{equation}
We now invoke \cite[Theorem 7.6.1]{Hobook}, which we formulate as follows: Let $A\in G\ell(\mathbb C,m)$ be such that $A^\star = A$ and $\Re A \ge 0$. Then 
\begin{equation}\label{gengaussi2}
\mathscr F\left(\frac{(4\pi)^{-\frac{m}{2}}}{\sqrt{\operatorname{det} A}} e^{- \frac{\sa A^{-1}\cdot,\cdot\da}{4}}\right)(\xi) =
e^{- 4 \pi^2  \sa A\xi,\xi\da},
\end{equation}
where $\sqrt{\operatorname{det} A}$ is the unique analytic branch such that $\sqrt{\operatorname{det} A}>0$ when $A$ is real. 
If in \eqref{gengaussi2} we take $A = i Q(t) I_m =  i e^{-it} \sin t\ I_m$ with  $t\in J^+$, then 
\[
A^{-1} = - i  \frac{e^{it}}{\sin t} I_m,
\]
and $\Re A = \sin^2 t\ I_m \ge 0$. We thus find 
\begin{equation}\label{larsetto2}
e^{-4\pi^2 i Q(t) |\xi|^2}  = \mathscr F\left(\frac{e^{\frac{imt}2}}{e^{\frac{i\pi m}{4}} (\sin t)^{\frac m2}}(4\pi
)^{-\frac{m}{2}} e^{i e^{it}\frac{|\cdot|^2}{4 \sin t}}\right)(\xi).
\end{equation}
From \eqref{hatv} and \eqref{larsetto2} we conclude that for every $x\in \Rm$ and $t\in J^+$
\begin{equation}\label{hatv2}
v(x,t) = \frac{e^{\frac{imt}2}}{e^{\frac{i\pi m}{4}} (\sin t)^{\frac m2}}(4\pi
)^{-\frac{m}{2}} \int_{\Rm} e^{i e^{it}\frac{|y-x|^2}{4 \sin t}} \vf(y) dy.
\end{equation} 
Finally, keeping \eqref{drift0} in mind, after some elementary algebraic manipulations, we obtain the representation \eqref{OUim} when $t\in J^+$. The part corresponding to $t\in J^-$ follows by elementary changes if one observes that now $A = e^{i\frac{3\pi}2} e^{-it} |\sin t| I_m$. 

\end{proof}

\begin{remark}
It may be of interest to compare \eqref{OUim} with the well-known representation\footnote{Formula   \eqref{gustavo} is classical. One way to easily obtain it is by taking $A = I_m, B = - 2 \sqrt \omega I_m, c = 0$ in (1.2) in the opening of \cite{Ho}.} 
\begin{align}\label{gustavo}
u(x,t) & = (4\pi)^{- \frac m2} e^{m t \sqrt \omega} \left(\frac{2\sqrt \omega}{\sinh(2t\sqrt \omega)}\right)^{\frac m2}
\\
& \times \int_{\Rm} \exp\left( -  \frac{\sqrt \omega}{2 \sinh(2t\sqrt \omega)} |e^{t\sqrt \omega} y - e^{-t\sqrt \omega} x|^2\right) \vf(y) dy
\notag
\end{align}
of the solution of the Cauchy problem for the Ornstein-Uhlenbeck operator
\begin{equation}\label{cpou}
\begin{cases}
u_t - \Delta u + 2 \sqrt \omega \langle x,\nabla u\rangle  = 0,\ \ \ \ \ \omega>0,
\\
u(x,0) = \vf(x).
\end{cases}
\end{equation}
 If one takes $\omega = \frac 14$, and keeping in mind that $\sinh it = i \sin t$, then it is clear that \emph{formally} substituting $t\to it$ in \eqref{gustavo}, one obtains the case $t\in J^+$ of \eqref{OUim}.
\end{remark}

In what follows we assume without restriction that $t\in J^+$. To further unravel \eqref{OUim}, and also to better clarify the role of the class $\mathscr K(\Rm)$ in \eqref{K}, note that expanding
\begin{equation}\label{expa}
\frac{|e^{it/2}y-e^{-it/2} x|^2}{4 \sin t} = \frac{e^{it}|y|^2 + e^{-it}|x|^2 - 2\sa x,y\da}{4 \sin t},
\end{equation}
we find 
\begin{equation}\label{OUim4}
f(x,t) = \frac{(4\pi
)^{-\frac{m}{2}}e^{\frac{imt}2}}{e^{\frac{i\pi m}{4}} (\sin t)^{\frac m2}} \int_{\Rm} e^{i \frac{e^{it}|y|^2 + e^{-it}|x|^2 - 2\sa x,y\da}{4 \sin t}} \vf(y) dy.
\end{equation}
The change of variable $y = 4 \pi \sin t\ z$ in the integral in \eqref{OUim4} gives
\begin{align*}
& \int_{\Rm} e^{- i \frac{\sa x,y\da}{2 \sin t}} e^{i \frac{e^{it}|y|^2 + e^{-it}|x|^2}{4 \sin t}} \vf(y) dy = (4 \pi \sin t)^m e^{i \frac{e^{-it}|x|^2}{4 \sin t}} \int_{\Rm} e^{- 2 \pi i \sa x,z\da} e^{i \frac{e^{it}  |4 \pi \sin t z|^2}{4 \sin t}} \vf(4 \pi \sin t z) dz
\\
& = (4 \pi \sin t)^m e^{\frac{|x|^2}{4}} e^{i \frac{\cot t |x|^2}{4}} \int_{\Rm} e^{- 2 \pi i \sa x,z\da} e^{i \frac{\cot t |4 \pi \sin t z|^2}{4}} e^{- \frac{|4 \pi \sin t z|^2}{4}} \vf(4 \pi \sin t z) dz.
\end{align*}
Keeping \eqref{psi} in mind, 
we can rewrite the above integral as 
\begin{align*}
& \int_{\Rm} e^{- i \frac{\sa x,y\da}{2 \sin t}} e^{i \frac{e^{it}|y|^2 + e^{-it}|x|^2}{4 \sin t}} \vf(y) dy = (4 \pi \sin t)^m e^{\frac{|x|^2}{4}} e^{i \frac{\cot t |x|^2}{4}} \mathscr F\left(\delta_{4 \pi \sin t}\ e^{i \frac{\cot t |\cdot|^2}4} \psi\right)(x)
\\
& = e^{\frac{|x|^2}{4}} e^{i \frac{\cot t |x|^2}{4}} \mathscr F\left(e^{i \frac{\cot t |\cdot|^2}4} \psi\right)(\frac{x}{4 \pi \sin t}),
\end{align*}
where we have denoted by $\delta_\la f(x) = f(\la x)$ the action of the dilation operator on a function $f$.
Going back to \eqref{OUim4}, we have finally established the  following basic result.

\begin{proposition}\label{P:semigroup}
Let $\vf\in \mathscr K(\Rm)$. Then for every $t\in J^+$ the function $f(x,t) = e^{i t L} \vf(x)$ is given by the formula
\begin{equation}\label{final}
e^{- \frac{|x|^2}{4}} f(x,t) =  (4\pi
)^{-\frac{m}{2}}\frac{e^{\frac{imt}2}}{e^{\frac{i\pi m}{4}} (\sin t)^{\frac m2}} e^{i \frac{\cot t |x|^2}{4}} \mathscr F\left(e^{i \frac{\cot t |\cdot|^2}4} \psi\right)(\frac{x}{4 \pi \sin t}),
\end{equation}
where $\psi$ is defined by \eqref{psi}.
\end{proposition}
Note that \eqref{final} shows that if $\vf\in \mathscr K(\Rm)$, then $f(\cdot,t)\in \mathscr K(\Rm)$, see Remark \ref{R:stable}, and also
\[
||f(\cdot,t)||_{L^2(\Rm,d\gamma)} = ||\vf||_{L^2(\Rm,d\gamma)}
\]
The equation \eqref{final} unveils the intertwining between the group $e^{i t L}$ and the Fourier transform. In the next section we exploit it to prove Theorem \ref{T:main}.
 

\section{Proof of Theorem \ref{T:main}}\label{S:pf}

In this section we prove the uncertainty principle in Theorem \ref{T:main}. We will also use Proposition \ref{P:connect} in Section \ref{S:ho} to derive Corollary \ref{C:main}. We will need the following result, see \cite[Theorem 1.1]{CEKPV}. We note for the reader that our normalisation of the Fourier transform
\[
\hat \vf(\xi) = \mathscr F(\vf)(\xi) = \int_{\Rm} e^{-2\pi i\sa \xi,x\da} \vf(x) dx,
\]
differs from theirs, and this accounts for the different constants in \eqref{hardyL2} below.

\begin{theorem}\label{T:har2}
Assume that $h:\Rm\to \R$ be a measurable function that satisfies 
\begin{equation}\label{hardyL2}
||e^{a|\cdot|^2} h||_{L^2(\Rm)} + ||e^{b|\cdot|^2} \hat h||_{L^2(\Rm)}<\infty.
\end{equation}
If $a b\ge \pi^2$, then $h\equiv 0$.  
\end{theorem} 

We are ready to give the 

\begin{proof}[Proof of Theorem \ref{T:main}]
Let $\vf\in L^2(\Rm,d\gamma)$ and $f(x,t) = e^{itL}\vf(x)$. Let $a, b>0$ be as in the statement of the theorem, so that $\sin s\not=0$. With $\psi$ as in \eqref{psi}, consider the function defined by
\begin{equation}\label{ht}
h_s(x) = e^{i \frac{\cot s |x|^2}4} \psi(x).
\end{equation}
We have
\begin{align}\label{one}
\int_{\Rm} e^{2a |x|^2} |h_s(x)|^2 dx & = \int_{\Rm} e^{2a |x|^2} |\psi(x)|^2 dx = \int_{\Rm} e^{2a |x|^2} |\vf(x)|^2 e^{- \frac{|x|^2}{2}}dx
\\
& =  ||e^{a |x|^2} f(\cdot,0)||^2_{L^2(\Rm,d\gamma)} < \infty,
\notag
\end{align}
in view of \eqref{L2}. On the other hand, \eqref{final} gives  
\begin{equation*}
\left|\hat h_t\left(\frac{x}{4 \pi \sin t}\right)\right| = (4\pi
)^{\frac{m}{2}} (\sin t)^{\frac m2} e^{- \frac{|x|^2}{4}} |f(x,t)|,
\end{equation*}
and therefore for every $t\in J^+$ we have
\begin{align}\label{beauty}
& \left(\int_{\Rm} e^{2b |x|^2} \left|\hat h_t\left(\frac{x}{4 \pi \sin t}\right)\right|^2 dx\right)^{1/2} 
 = (4\pi
)^{\frac{m}{2}} (\sin t)^{\frac m2} \left(\int_{\Rm} e^{2b |x|^2} |f(x,t)|^2 e^{- \frac{|x|^2}{2}} dx\right)^{1/2}
\\
& = (4\pi
)^{\frac{m}{2}} (\sin t)^{\frac m2} ||e^{b |x|^2} f(\cdot,t)||_{L^2(\Rm,d\gamma)}.
\notag
\end{align}
If $s\in J^+$ is such that $||e^{b |x|^2} f(\cdot,s)||_{L^2(\Rm,d\gamma)} < \infty$, see \eqref{L2}, then we infer from \eqref{beauty} that 
\begin{equation}\label{nik}
\int_{\Rm} e^{2(16 b \pi^2 \sin^2 s) |y|^2} |\hat h_s(y)|^2 dy < \infty.
\end{equation}
In view of \eqref{one} and \eqref{nik}, applying Theorem \ref{T:har2} to the function $h_s$ we conclude that, if 
\[
16 \pi^2 a b  \sin^2 s \ge \pi^2\ \Longleftrightarrow\  a b \sin^2 s  \ge \frac{1}{16},
\]
then $h_s(x) = 0$ for every $x\in \Rm$. From \eqref{ht}, it is clear that this implies $\psi \equiv 0$, and therefore $\phi \equiv 0$, in $\Rm$.

\end{proof}

Next, we present the 

\begin{proof}[Proof of Corollary \ref{C:main}]
Let $u_0\in Dom(H)\subset L^2(\Rm)$, and let $u(x,t) = e^{itH} u_0(x)$. Suppose that $a, b>0$ are such that \eqref{L22} be satisfied. For every $t\in \R$ such that $\sin t\not =0$, define
\[
h_t(x) = e^{i \frac{\cot t |x|^2}4} u_0(x).
\]
It is clear that
\[
||e^{a|\cdot|^2} h_t||_{L^2(\Rm)} = ||e^{a|\cdot|^2} u_0||_{L^2(\Rm)} <\infty.
\]
Arguing as in the proof of Theorem \ref{T:main}, but this time using \eqref{finalino} in Corollary \ref{C:semigroup}, we infer that
\[
\left(\int_{\Rm} e^{2b |x|^2} \left|\hat h_s\left(\frac{x}{4 \pi \sin s}\right)\right|^2 dx\right)^{1/2} \cong ||e^{b|\cdot|^2} u(\cdot,s)||_{L^2(\Rm)}  < \infty.
\]
Again by Theorem \ref{T:main} we conclude that, under the hypothesis $a b \sin^2 s \ge \frac{1}{16}$, we must have $u_0\equiv 0$.

\end{proof}

We close this section by noting that, with Proposition \ref{P:semigroup} in hand, the proof of Proposition \ref{P:disp} follows directly from \eqref{final} and from Beckner's sharp version of the the Hausdorff-Young theorem, see \cite{Be}. We leave the relevant details to the interested reader. We will  return to more general dispersive estimates for the group $e^{itL}$ in a future study.


\section{Appendix: The imaginary harmonic oscillator}\label{S:ho}

In what follows we consider the harmonic oscillator in $\Rm$
\begin{equation}\label{H}
H = \Delta - \frac{|x|^2}4
\end{equation}
and the Cauchy problem in $\Rm\times (0,\infty)$ for the associated Schr\"odinger operator
\begin{equation}\label{cpHO}
\begin{cases}
\p_t u - i H u = 0
\\
u(x,0) = u_0(x),
\end{cases}
\end{equation}
where the initial datum $u_0$ will be taken e.g. in $\mathscr S(\Rm)$. The following lemma establishes a general principle, one interesting consequence of which is that it allows to connect \eqref{cpHO} to the problem \eqref{cp0}, and in fact show that they are equivalent. The functions $\Phi$ and $h$ in its statement are assumed complex-valued.

\begin{lemma}\label{L:ouho}
Let $\Phi\in C(\R^{m+1})$ and $h\in C^2(\R^{m+1})$ be connected by the following  nonlinear Schr\"odinger equation
\begin{equation}\label{riccati}
i h_t + \Delta h - |\nabla h|^2  = \Phi.
\end{equation}
Then $u$ solves the partial differential equation
\begin{equation}\label{pdeho}
P u = i(\Delta u + \Phi u) - u_t = 0
\end{equation}
if and only if $f$ defined by the transformation
\begin{equation}\label{genexp}
u(x,t) = e^{-h(x,t)} f(x,t),
\end{equation} 
solves the equation
\begin{equation}\label{PDEou}
i(\Delta f - 2 \langle\nabla h,\nabla f\rangle) - f_t = 0.
\end{equation}
\end{lemma}

\begin{proof}
With $u$ as in \eqref{genexp}, we find
\begin{align*}
P u & = i \Delta(e^{-h} f) + i \Phi e^{-h} f - (e^{-h} f)_t
\\
& = i f \Delta(e^{-h}) + i e^{-h} \Delta f + 2 i \langle\nabla(e^{-h}),\nabla f\rangle + i \Phi e^{-h} f
 - (e^{-h})_t f -  e^{-h} f_t
\\
& = i e^{-h} f |\nabla h|^2 - i e^{-h} f \Delta h + i e^{-h} \Delta f - 2 i e^{-h} \langle\nabla h,\nabla f\rangle + i \Phi e^{-h} f
 - (e^{-h})_t f -  e^{-h} f_t
 \\
 & = e^{-h}\left\{i \left[\Phi - i h_t - (\Delta h - |\nabla h|^2 )\right] f + i \Delta f - 2 i \langle\nabla h,\nabla f\rangle - f_t\right\}.
\end{align*}
This computation proves that if $h$ and $\Phi$ solve \eqref{riccati}, then $u$ solves \eqref{pdeho} if and only if $f$ is a solution of \eqref{PDEou}.

\end{proof}

With Lemma \ref{L:ouho} in hand, we now return to \eqref{H} and prove the following result.
\begin{proposition}\label{P:connect}
A function $u$ solves the Cauchy problem \eqref{cpHO} if and only if the function
\begin{equation}\label{fu}
f(x,t) =  u(x,t)\ e^{\frac{|x|^2}4 + i \frac m2 t}
\end{equation}
solves \eqref{cp0} with $f(x,0) = \vf(x) = u_0(x)\ e^{\frac{|x|^2}4}$.
\end{proposition}

\begin{proof}
It is clear from \eqref{pdeho} that, in order to obtain from it the PDE in \eqref{cpHO}, we need  $\Phi(x,t) = - \frac{|x|^2}4$. With this choice, we look for a function $h(x,t)$ that is connected to such $\Phi$ by the equation \eqref{riccati}. A natural ansatz is $h(x,t) = A |x|^2 + B t$, with $A, B\in \mathbb C$ to be determined. Since $\Delta h = 2m A$, $h_t = B$ and $|\nabla h|^2 = 4 A^2 |x|^2$, to satisfy \eqref{riccati} we want
\[
i B + 2m A - 4 A^2 |x|^2 = - \frac{|x|^2}4,
\]
which holds iff $A = \frac{1}4, B =  i \frac m2$, and thus
\begin{equation}\label{h}
h(x,t) = \frac{|x|^2}4 + i \frac m2 t.
\end{equation}
With such choice of $h(x,t)$, the equation \eqref{genexp} in Lemma \ref{L:ouho} shows that $f(x,t)$ defined in \eqref{fu} solves the Cauchy problem \eqref{cp0}, with $f(x,0) = \vf(x) = u_0(x)\ e^{\frac{|x|^2}4}$. The ``if and only if" character of the statement is obvious.

\end{proof}

If $u_0\in \So$, then it is clear that $e^{\frac{|\cdot|^2}4} u_0\in \mathscr K(\Rm)$. According to Proposition \ref{P:connect}, we can express the group $e^{itH}$ by the formula
\begin{equation}\label{eitH}
e^{itH} u_0(x) = e^{-\frac{|x|^2}4 - i \frac m2 t} e^{itL}(e^{\frac{|\cdot|^2}4} u_0)(x).
\end{equation}
Applying \eqref{OUim}, we infer from \eqref{eitH}.

\begin{corollary}\label{C:HOim}
Given $u_0\in \So$, for $t\in J^+$ one has
\begin{equation}\label{eitHg}
e^{itH} u_0(x) =  \frac{(4\pi
)^{-\frac{m}{2}} e^{-\frac{|x|^2}4}}{e^{\frac{i\pi m}{4}} (\sin t)^{\frac m2}} \int_{\Rm} e^{i \frac{|e^{it/2} y-e^{-it/2} x|^2}{4 \sin t}} e^{\frac{|y|^2}4} u_0(y) dy.
\end{equation}
\end{corollary}
Using \eqref{expa} in \eqref{eitHg}, we thus obtain the following counterpart of Proposition \ref{P:semigroup}.

\begin{corollary}\label{C:semigroup}
Given $u_0\in \mathscr S(\Rm)$, let $u(x,t) = e^{itH}u_0(x)$. Then for every $t\in J^+$ one has
\begin{equation}\label{finalino}
u(x,t) =  \frac{(4\pi
)^{-\frac{m}{2}}}{e^{\frac{i\pi m}{4}} (\sin t)^{\frac m2}} e^{i \frac{\cot t |x|^2}{4}} \mathscr F\left(e^{i \frac{\cot t |\cdot|^2}4} u_0\right)(\frac{x}{4 \pi \sin t}).
\end{equation}
\end{corollary} 


\vskip 0.2in

\section{Declarations}

\noindent \textbf{Data availability statement:} This manuscript has no associated data.

\vskip 0.2in

\noindent \textbf{Funding and/or Conflicts of interests/Competing interests statement:} The author declares that he does not have any conflict of interest for this work

\bibliographystyle{amsplain}

\end{document}